\documentclass[12pt, reqno]{amsart}
\usepackage{amsmath, amsthm, amscd, amsfonts, latexsym, amssymb, graphicx, color, enumerate}
\usepackage[colorlinks,linkcolor=,citecolor = blue]{hyperref}

\makeatletter \oddsidemargin.9375in \evensidemargin \oddsidemargin
\newtheorem{theorem}{Theorem}[section]
\newtheorem{lemma}[theorem]{Lemma}
\newtheorem{proposition}[theorem]{Proposition}
\newtheorem{corollary}[theorem]{Corollary}
\theoremstyle{definition}

\newtheorem{example}[theorem]{Example}

\theoremstyle{remark}
\newtheorem{remark}[theorem]{Remark}
\numberwithin{equation}{section}

\newcommand{\pr}[1]{\mathfrak{#1}}
\newcommand{\prodpdots}[1]{{{\mathfrak{p}}_1}, \dots ,{{\mathfrak{p}}_{#1}}}
\newcommand{\assr}[1]{\mathrm{Ass}(#1)}
\newcommand{\qqq}[1]{{{\mathfrak{p}}_1}^{r_1} \cdots {{\mathfrak{p}}_{#1}}^{r_{#1}}}
\newcommand{\prp}[1]{{{\mathfrak{p}}_1} \cdots {{\mathfrak{p}}_{#1}}}

\begin{document}

\setcounter{page}{1}

\title[Product of prime ideals as factorization of submodules]{Product of prime ideals as factorization of submodules}

\author[Thulasi, Duraivel, and Mangayarcarassy]{K. R. Thulasi$^{*}$, T. Duraivel, and S. Mangayarcarassy}

\thanks{{The first author was supported by INSPIRE Fellowship (IF170488) of the Department of Science and Technology (DST), Government of India.\\}{\scriptsize
\hskip -0.4 true cm MSC(2010): Primary: 13A05; Secondary: 13A15, 13E05, 13E15
$*$Corresponding author }}
\begin{abstract}
For a proper submodule $N$ of a finitely generated module $M$ over a Noetherian ring, the product of prime ideals which occur in a regular prime extension filtration of $M$ over $N$ is defined as its generalized prime ideal factorization in $M$. In this article, we find conditions for a product of prime ideals to be the generalized prime ideal factorization of a submodule of some module. We show that a power of a prime ideal occurs in a generalized prime ideal factorization only if it is not equal to its lesser powers. Also, we show that ${{\mathfrak{p}}_1}^{r_1} \cdots {{\mathfrak{p}}_{n}}^{r_{n}}$ is a generalized prime ideal factorization if and only if for each $1 \leq i \leq n$, ${{\mathfrak{p}}_i}^{r_i}$ is the generalized prime ideal factorization of some submodule of a module.
\end{abstract}

\maketitle
\section{Introduction}

Throughout this article $R$ is a commutative Noetherian ring with identity and all $R$-modules are assumed to be finitely generated and unitary.

Let $N$ be a proper submodule of an $R$-module $M$. If $K$ is a submodule of $M$ such that $N$ is a $\pr p$-prime submodule of $K$, then we say $K$ is a $\pr p$-prime extension of $N$ in $M$ and denote it as $N \overset{\pr p} \subset K$. In this case, $\assr{K/N} = \{\pr p\}$ \cite[Theorem~1]{CPLu}. A $\pr p$-prime extension $K$ of $N$ is said to be maximal in $M$ if there is no $\pr p$-prime extension $L$ of $N$ in $M$ such that $L \supset K$. It is proved that if $N$ is a proper submodule of $M$ and $\pr p$ is a maximal element in $\assr{M/N}$, then $(N : \pr p)$ is the unique maximal $\pr p$-prime extension of $N$ in $M$ \cite[Theorem~11]{A} and it is called a regular $\pr p$-prime extension of $N$ in $M$.

A filtration of submodules $\mathcal{F} : N = M_0 \overset{{\pr p}_1}\subset M_1 \subset \cdots \overset{{\pr p}_n}\subset M_n = M$ is called a regular prime extension (RPE) filtration of $M$ over $N$ if each $M_i$ is a regular ${\pr p}_i$-prime extension of $M_{i-1}$ in $M$, $1 \leq i \leq n$. We also have that $M_i \overset{{\pr p}_{i+1}}\subset M_{i+1} \subset \cdots \overset{{\pr p}_j}\subset M_j$ is an RPE filtration of $M_j$ over $M_i$ for every $0 \leq i < j \leq n$. RPE filtrations are defined and studied in \cite{A} and it is also noted that RPE filtrations are weak prime decompositions defined by Dress in \cite{Dress}. The set of prime ideals which occur in the RPE filtration $\mathcal{F}$ is precisely $\assr{M/N}$. This is proved more generally in the following lemma.
\begin{lemma}\cite[Proposition~14]{A}\label{lemma*}
Let $ N = M_0 \overset{{\pr p}_1}\subset M_1 \subset \cdots \overset{{\pr p}_n}\subset M_n = M $ be a filtration of submodules such that each $M_{i-1} \overset{{\pr p}_i}\subset M_i$ is a maximal ${\pr p}_i$-prime extension. Then $\assr{M/M_{i-1}} = \{{\pr p}_{i},\ldots,{\pr p}_n\}$ for $1 \leq i \leq n$. In particular, $\assr{M/N} = \{{\pr p}_1,\ldots,{\pr p}_n\}$.
\end{lemma}
The submodules occurring in an RPE filtration are characterized as below.
\begin{lemma}\cite[Lemma~3.1]{B}
Let $N$ be a proper submodule of an $R$-module $M$. If $N = M_0 \overset{{\pr p}_1}\subset M_1 \subset \cdots \overset{{\pr p}_n}\subset M_n = M$ is an RPE filtration of $M$ over $N$, then $M_i = \{ x \in M \mid \prp i x \subseteq N \}= (N : \prp i)$ for $ 1 \leq i \leq n$. \label{lemma1}
\end{lemma}

It is proved that the number of times a prime ideal occurs in any RPE filtration of $M$ over $N$ is unique \cite[Theorem~22]{A}. So if $N = M_0 \overset{{\pr p}_1}\subset M_1 \subset \cdots \overset{{\pr p}_n}\subset M_n = M $ is an RPE filtration, then the product $\prp n$ is uniquely defined for $N$ in $M$ and it is called the generalized prime ideal factorization of $N$ in $M$, written as ${\mathcal{P}}_M(N)= \prp n$. Generalized prime ideal factorization of submodules is defined and its various properties are studied in \cite{E}.

In general, a product of prime ideals need not be a generalized prime ideal factorization of some submodule of a module.
\begin{example}\label{example4.1}
Let $R = \frac{k[x,y,z]}{(xy - z^2 , x^2 - yz)}$ and $\pr p$ be the prime ideal $(\overline{x}, \overline{z})$. Suppose there exists an $R$-module $M$ and a submodule $N$ with ${\mathcal{P}}_M(N) = {\pr p}^2$. Then there is an RPE filtration $N \overset{\pr p} \subset N_1 \overset{\pr p} \subset M$. By Lemma \ref{lemma*}, $\assr{M/N_1} = \{ \pr p\}$. So $\pr p = (N_1 : m)$ for some $m \in M$. We have ${\pr p}(\overline{x}, \overline{y}, \overline{z}) = {\pr p}^2$, which implies ${\pr p}(\overline{x}, \overline{y}, \overline{z})M = {\pr p}^2M \subseteq N$. Therefore, $(\overline{x}, \overline{y}, \overline{z})m \subseteq (N : \pr p) = N_1$ [Lemma \ref{lemma1}] which implies $(\overline{x}, \overline{y}, \overline{z}) \subseteq (N_1 : m) = \pr p$, a contradiction. Therefore, we cannot have an $R$-module $M$ and a submodule $N$ with ${\mathcal{P}}_M(N) = {\pr p}^2$.
\end{example}

In this article, for a product of prime ideals ${\pr p}_1, \dots, {\pr p}_n$, we find conditions for the existence of an $R$-module $M$ and a submodule $N$ of $M$ with ${\mathcal{P}}_M(N) = \prp n$. We show that a power of a prime ideal occurs in a generalized prime ideal factorization only if it is not equal to its lesser powers. We show that ${{\pr p}_i}^{r_i} \neq {{\pr p}_i}^{r_i - 1}$ and $\assr{R/{{\pr p}_i}^{r_i}} = \{{{\pr p}_i}\}$ for every $1 \leq i \leq n$ is a sufficient condition for the existence of an $R$-module $M$ and a submodule $N$ of $M$ with ${\mathcal{P}}_{M}(N)=\qqq n$. Also, we show that $\qqq n$ is a generalized prime ideal factorization if and only if for each $1 \leq i \leq n$, ${{\pr p}_i}^{r_i}$ is the generalized prime ideal factorization of some submodule of a module.

\section{Conditions for ideals to be generalized prime ideal factorization of submodules}
\label{Sec:2}
We need the following lemmas.

\begin{lemma}\label{reg_div_rpe}
Let $N$ be a proper submodule of an $R$-module $M$. If a submodule $K$ of $M$ lies in an RPE filtration of $M$ over $N$, then ${\mathcal{P}}_M(N) = {\mathcal{P}}_M(K) {\mathcal{P}}_K(N)$.
\end{lemma}
\begin{proof}
If a submodule $K$ of $M$ lies in an RPE filtration of $M$ over $N$, then we have an RPE filtration
$$N = N_0 \overset{{\pr p}_1}\subset N_1 \subset \cdots  \overset{{\pr p}_r}\subset N_r \overset{{\pr p}_{r+1}} \subset \cdots \overset{{\pr p}_n} \subset N_n = M,$$ 
where $N_r = K$ for some $r$. So ${\mathcal{P}}_M(N) = \prp n$. Then $$N = N_0 \overset{{\pr p}_1}\subset N_1 \subset \cdots  \overset{{\pr p}_r}\subset N_r=K \text{ and } K=N_r \overset{{\pr p}_{r+1}} \subset N_{r+1} \subset \cdots \overset{{\pr p}_n} \subset N_n = M$$ are RPE filtrations, and therefore, ${\mathcal{P}}_K(N) = \prp r$ and ${\mathcal{P}}_M(K)={\pr p}_{r+1} \cdots {\pr p}_n $. This proves the lemma.
\end{proof}

\begin{lemma}\cite[Lemma~20]{A}\label{interchange}
Let $N$ be a proper submodule of $M$ and $N = M_0 \subset \cdots \subset M_{i-1}\overset{{\pr p}_i}\subset M_i \overset{{\pr p}_{i+1}}\subset M_{i+1}\subset \cdots \subset M_n = M $ be an RPE filtration of $M$ over $N$. If ${\pr p}_{i+1} \not\subseteq {\pr p}_{i}$ for some $i$, then there exists a submodule $K_i$ of $M$ such that $N = M_0 \subset \cdots \subset M_{i-1}\overset{{\pr p}_{i+1}}\subset K_i \overset{{\pr p}_{i}}\subset M_{i+1}\subset \cdots \subset M_n$ $=$ $M$ is an RPE filtration of $M$ over $N$.
\end{lemma}
\begin{remark}\label{reordering}
Suppose a product of prime ideals $\prp r$ is ${\mathcal{P}}_M(N)$ for some module $M$ and a submodule $N$ of $M$. For every reordering of $\prodpdots r$ with ${\pr p}_i \not \subset {\pr p}_j$ for $i<j$, by applying Lemma \ref{interchange} for a sufficient number of times, we can get an RPE filtration $N \overset{{\pr p}_1}\subset N_1 \overset{{\pr p}_2}\subset N_2 \subset \cdots \subset N_{r-1} \overset{{\pr p}_r}\subset N_r = M$.
\end{remark}

\begin{proposition}\label{idealthm}
For a prime ideal $\,\pr p$ and a positive integer $r$, if $\,{\pr p}^r$ is the generalized prime ideal factorization of a submodule of an $R$-module, then ${\pr p}^r \neq {\pr p}^{r - 1}$.
\end{proposition}
\begin{proof}
Suppose $M$ is an $R$-module and $N$ is a submodule of $M$ with ${\mathcal{P}}_M(N) = {\pr p}^r$. Then we have an RPE filtration
$$N \overset{\pr p}\subset N_1 \overset{\pr p}\subset N_2 \subset \cdots \subset N_{r-1} \overset{\pr p}\subset N_r = M.$$
If ${\pr p}^r = {\pr p}^{r - 1}$, then by Lemma \ref{lemma1}, $M= N_r = \{x \in M \mid {\pr p}^r x \subseteq N \} = \{x \in M \mid {\pr p}^{r - 1} x \subseteq N \} = N_{r-1}$, a contradiction.
\end{proof}
Example \ref{example4.1} shows that the converse of the above result is not true.

\begin{corollary}\label{cor_power_iff}
Let $\pr p$ be a prime ideal in $R$ and $r$ be a positive integer. Then ${\mathcal{P}}_R({{\pr p}}^{r}) = {{\pr p}}^{r}$ if and only if ${\pr p}^r \neq {\pr p}^{r - 1}$ and $\assr{R/{{\pr p}^r}} = \{ \pr p \}$.
\end{corollary}
\begin{proof}
If ${\pr p}^r \neq {\pr p}^{r - 1}$ and $\assr{R/{{\pr p}^r}} = \{ \pr p \}$, then by \cite[Corollary~2.16]{E}, ${\mathcal{P}}_R({{\pr p}}^{r}) = {{\pr p}}^{r}$. Conversely if ${\mathcal{P}}_R({{\pr p}}^{r}) = {{\pr p}}^{r}$, then we have an RPE filtration ${{\pr p}}^{r} \overset{\pr p}\subset {\pr a}_1 \overset{\pr p}\subset {\pr a}_2 \subset \cdots \subset {\pr a}_{r-1} \overset{\pr p}\subset {\pr a}_r = R$. By Lemma \ref{lemma*}, $\assr{R/{{\pr p}^r}} = \{ \pr p \}$ and by Proposition \ref{idealthm}, ${\pr p}^r \neq {\pr p}^{r - 1}$.
\end{proof}

\begin{corollary}\label{corprimepowers}
Let ${\pr p}_1, \dots, {\pr p}_n$ be distinct prime ideals in $R$ and $r_1, \dots , r_n$ be positive integers. If there exists an $R$-module $M$ and a submodule $N$ of $M$ with ${\mathcal{P}}_M(N) = \qqq n$, then ${{\pr p}_i}^{r_i} \neq {{\pr p}_i}^{r_i - 1}$ for every $1 \leq i \leq n$. 
\end{corollary}
\begin{proof}
Reorder ${\pr p}_1, \dots, {\pr p}_n$ so that ${\pr p}_i \not \subset {\pr p}_j$ for $i<j$. By Remark \ref{reordering}, we have an RPE filtration
\begin{multline*}
N= L^{(0)} \overset{{\pr p}_{1}}\subset L^{(1)}_1 \overset{{\pr p}_{1}}\subset L^{(1)}_2 \overset{{\pr p}_{1}}\subset \cdots \overset{{\pr p}_{1}}\subset L^{(1)}_{r_1} \overset{{\pr p}_{2}}\subset L^{(2)}_1 \subset \cdots \\ \overset{{\pr p}_{n-1}}\subset L^{(n-1)}_{r_{n-1}} \overset{{\pr p}_{n}}\subset L^{(n)}_1 \overset{{\pr p}_{n}}\subset L^{(n)}_2 \subset \cdots \overset{{\pr p}_{n}}\subset L^{(n)}_{r_{n}} = M
\end{multline*}
of $M$ over $N$. Then for every $1 \leq i \leq n$, $$L^{(i-1)}_{r_{i-1}} \overset{{\pr p}_{i}}\subset L^{(i)}_1 \overset{{\pr p}_{i}}\subset L^{(i)}_2 \subset \cdots \overset{{\pr p}_{i}}\subset L^{(i)}_{r_{i}-1} \overset{{\pr p}_{i}}\subset L^{(i)}_{r_{i}}$$ is an RPE filtration of $L^{(i)}_{r_{i}}$ over $L^{(i-1)}_{r_{i-1}}$, and therefore, ${{\pr p}_i}^{r_i}$ is the generalized prime ideal factorization of $L^{(i-1)}_{r_{i-1}}$ in $L^{(i)}_{r_{i}}$. By Proposition \ref{idealthm}, ${{\pr p}_i}^{r_i} \neq {{\pr p}_i}^{r_i - 1}$.
\end{proof}

\begin{theorem}\label{idealthm1}
Let ${\pr p}_1, \dots, {\pr p}_n$ be distinct prime ideals in $R$ and $r_1, \dots , r_n$ be positive integers. If there exists an $R$-module $M$ and a submodule $N$ of $M$ with ${\mathcal{P}}_M(N) = \qqq n$, then ${{\pr p}_1}^{r_1} \cdots {{\pr p}_i}^{r_i - 1} \cdots {{\pr p}_n}^{r_n} \neq \qqq n$ whenever ${\pr p}_i$ is minimal among $\{{\pr p}_1, \dots, {\pr p}_n\}$.
\end{theorem}
\begin{proof}
Let ${\pr p}_i$ be minimal among $\{{\pr p}_1, \dots, {\pr p}_n\}$. We can reorder ${\pr p}_1, \dots, {\pr p}_n$ such that ${\pr p}_n = {\pr p}_i$ and ${\pr p}_j \not \subset {\pr p}_k$ for $j<k$. So without loss of generality, we assume that $i=n$. Then by Remark \ref{reordering}, we have the RPE filtration
\begin{multline*}
N= L^{(0)} \overset{{\pr p}_{1}}\subset L^{(1)}_1 \overset{{\pr p}_{1}}\subset L^{(1)}_2 \overset{{\pr p}_{1}}\subset \cdots \overset{{\pr p}_{1}}\subset L^{(1)}_{r_1} \overset{{\pr p}_{2}}\subset L^{(2)}_1 \subset \cdots \\ \overset{{\pr p}_{n-1}}\subset L^{(n-1)}_{r_{n-1}} \overset{{\pr p}_{n}}\subset L^{(n)}_1 \overset{{\pr p}_{n}}\subset L^{(n)}_2 \subset \cdots \overset{{\pr p}_{n}}\subset L^{(n)}_{r_{n}-1} \overset{{\pr p}_{n}}\subset L^{(n)}_{r_{n}} = M
\end{multline*}
of $M$ over $N$. Then for every $x \in M$, $\qqq n x \subseteq N$ [Lemma \ref{lemma1}]. Suppose ${{\pr p}_1}^{r_1} \cdots {{\pr p}_{n-1}}^{r_{n-1}} {{\pr p}_n}^{r_n - 1} = \qqq n$. Then ${{\pr p}_1}^{r_1} \cdots {{\pr p}_{n-1}}^{r_{n-1}}$ ${{\pr p}_n}^{r_n - 1} x  \subseteq N$. By Lemma \ref{lemma1}, $x \in L^{(n)}_{r_{n}-1}$. This implies $L^{(n)}_{r_{n}-1} = M$, which is not true. Therefore, ${{\pr p}_1}^{r_1} \cdots {{\pr p}_i}^{r_i - 1} \cdots {{\pr p}_n}^{r_n} \neq \qqq n$.
\end{proof}

\begin{remark}
The above result is not true if ${\pr p}_i$ is not minimal among $\{{\pr p}_1, \dots, {\pr p}_n\}$. For example, in the ring $R = \frac{k[x,y,z]}{(xy-z, yz-x)}$, let $\pr p$ and $\pr q$ be the prime ideals $(\overline{x}, \overline{y}, \overline{z})$ and $(\overline{x}, \overline{z})$ respectively. In this case, $\pr p \pr q = \pr q$. For the $R$-module $M = \frac{R}{\pr q} \oplus \frac{R}{\pr q}$ and its submodule $N = \frac{\pr p}{\pr q} \oplus 0$, we have the RPE filtration $$N = \dfrac{\pr p}{\pr q} \oplus 0 \quad \overset{\pr p} \subset \quad \dfrac{R}{\pr q} \oplus 0 \quad \overset{\pr q}\subset \quad \dfrac{R}{\pr q} \oplus \dfrac{R}{\pr q} = M$$ of $M$ over $N$. So we have ${\mathcal{P}}_M(N) = \pr p \pr q$ and $\pr p \pr q = \pr q$.
\end{remark}

The next lemma is about the regular prime extensions on direct sums.
\begin{lemma}\label{directsumlemma}
Let $N$, $N'$ be submodules of modules $M$, $M'$ respectively and $N \overset{\pr p} \subset K$ be a regular ${\pr p}$-prime extension in $M$.
\begin{enumerate}[(i)]
\item If $N' \overset{\pr p} \subset K'$ is a regular ${\pr p}$-prime extension in $M'$, then $N \oplus N' \overset{\pr p} \subset K \oplus K'$ is a regular ${\pr p}$-prime extension in $M \oplus M'$.
\item If ${\pr p} \notin \assr{M'/N'}$ and ${\pr p}$ is a maximal element in $\assr{M/N} \cup \assr{M'/N'}$, then $N \oplus N' \overset{\pr p} \subset K \oplus N'$ is a regular ${\pr p}$-prime extension in $M \oplus M'$.
\end{enumerate}
\end{lemma}
\begin{proof}
\emph{(i)} We have ${\pr p}K \subseteq N$ and ${\pr p}K' \subseteq N'$. Therefore, ${\pr p}\subseteq (N \oplus N' : K \oplus K')$. Let $a \in (N \oplus N' : K \oplus K')$ and $x \in K \setminus N$. Then $(x,0) \in K \oplus K' \setminus N \oplus N'$ and $a(x,0) \in N \oplus N'$. This implies $ax \in N$. Since $x \notin N$ and $N \overset{\pr p} \subset K$ is a ${\pr p}$-prime extension in $M$, we get $a \in \pr p$. Hence, $(N \oplus N' : K \oplus K') = \pr p$.

Let $a \in R$, $(x,y) \in K \oplus K'$ such that $a(x,y) \in N \oplus N'$. Suppose $(x,y) \notin N \oplus N'$. Then either $x \notin N$ or $y \notin N'$. If $x \notin N$, since $ax \in N$ and $N \overset{\pr p} \subset K$ is a ${\pr p}$-prime extension, we get $a \in \pr p$. If $x \in N$, then $y \notin N'$ and since $N' \overset{\pr p} \subset K'$ is a ${\pr p}$-prime extension in $M'$, $ay \in N'$ implies $a \in \pr p$. Hence, $N \oplus N' \overset{\pr p} \subset K \oplus K'$ is a ${\pr p}$-prime extension in $M \oplus M'$.

Suppose it is not maximal. Let $L \oplus L'$ be a ${\pr p}$-prime extension of $N \oplus N'$ in $M \oplus M'$ with $K \oplus K' \subsetneq L \oplus L'$. Then ${\pr p}(L \oplus L') \subseteq N \oplus N'$. This implies ${\pr p}L \subseteq N$. Since $N \subsetneq K \subset L$, we have $N \subsetneq L$. Let $x \in L \setminus N$. Then for every $a \in (N:L)$, $ax \in N$, and therefore, $a(x,0) \in N \oplus N'$. Since $(x,0) \notin N \oplus N'$, we get $a \in \pr p$. Therefore, $L$ is a $\pr p$-prime extension of $N$ in $M$ and $K \subset L$. Since $N \overset{\pr p} \subset K$ is a maximal ${\pr p}$-prime extension in $M$, we get $K=L$. Similarly for $N' \subsetneq K' \subset L'$, we get $K' = L'$. This implies $K \oplus K' = L \oplus L'$, a contradiction. Hence, $N \oplus N' \overset{\pr p} \subset K \oplus K'$ is a maximal ${\pr p}$-prime extension in $M \oplus M'$.

Since $\pr p$ is a maximal element in both $\assr{M/N}$ and $\assr{M'/N'}$, $\pr p$ is maximal in $\assr{M/N} \cup \assr{M'/N'}$. Therefore, $\pr p$ is maximal in $\assr{\frac{M \oplus M'}{N \oplus N'}}$. Hence, $N \oplus N' \overset{\pr p} \subset K \oplus K'$ is a regular ${\pr p}$-prime extension in $M \oplus M'$.

\emph{(ii)} Taking $K' = N'$ in the proof of $(i)$, we get $N \oplus N' \overset{\pr p} \subset K \oplus N'$ is a ${\pr p}$-prime extension in $M \oplus M'$. Suppose this ${\pr p}$-prime extension is not maximal. Let $L \oplus L'$ be a ${\pr p}$-prime extension of $N \oplus N'$ in $M \oplus M'$ with $K \oplus N' \subsetneq L \oplus L'$. Then ${\pr p}(L \oplus L') \subseteq N \oplus N'$. This implies ${\pr p}L \subseteq N$. Since $N \subsetneq K \subset L$, we have $N \subsetneq L$. As in the proof of $(i)$, we get $K=L$. We have $N' \subseteq L'$. If $N' \subsetneq L'$, then using the same argument we get $N' \overset{\pr p}\subset L'$ is a ${\pr p}$-prime extension in $M'$ and therefore, ${\pr p} \in \assr{M'/N'}$ [Lemma \ref{lemma*}], which is not the case. Therefore, $N' = L'$. Then $K \oplus N' = L \oplus L'$, which is a contradiction. Hence, $N \oplus N' \overset{\pr p} \subset K \oplus N'$ is a maximal ${\pr p}$-prime extension in $M \oplus M'$.

Since $\pr p$ is a maximal element in $\assr{M/N} \cup \assr{M'/N'}$, $\pr p$ is maximal in $\assr{\frac{M \oplus M'}{N \oplus N'}}$. Hence, $N \oplus N' \overset{\pr p} \subset K \oplus N'$ is a regular ${\pr p}$-prime extension in $M \oplus M'$.
\end{proof}

\begin{lemma}\label{directsum_corollary}
Let $N_1, \dots , N_n$ be submodules of $R$-modules $M_1, \dots , M_n$ respectively and let ${\pr p}$ be a maximal element in $\cup_{i=1}^n\assr{M_i/N_i}$. For $1 \leq i \leq n$, let $K_i$ be the regular ${\pr p}$-prime extension of $N_i$ in $M_i$ if $\pr p \in \assr{M_i/N_i}$ and $K_i = N_i$ otherwise. Then ${\oplus_{i=1}^n{N}_i} \overset{\pr p} \subset {\oplus_{i=1}^n{K}_i}$ is a regular ${\pr p}$-prime extension in ${\oplus_{i=1}^n{M}_i}$.
\end{lemma}
\begin{proof}
Since $\mathrm{Ass}\Big({\dfrac{\oplus_{i=1}^n{M}_i}{\oplus_{i=1}^n{N}_i}}\Big) = \bigcup\limits_{i=1}^n{\mathrm{Ass}\Big(\dfrac{M_i}{N_i} \Big)}$ the proof follows by induction on $n$ using Lemma \ref{directsumlemma}.
\end{proof}

\begin{proposition}\label{direct_sum_max}
Let $N$, $N'$ be proper submodules of modules $M$, $M'$ respectively. If ${\mathcal{P}}_{M}(N) = \qqq k$ and ${\mathcal{P}}_{M'}(N') = {{\mathfrak{p}}_1}^{s_1} \cdots {{\mathfrak{p}}_{k}}^{s_k}$ for some non-negative integers $r_i$ and $s_i$ with either $r_i >0$ or $s_i >0$, then ${\mathcal{P}}_{M \oplus M'}(N \oplus N') = {{\mathfrak{p}}_1}^{t_1} \cdots {{\mathfrak{p}}_{k}}^{t_k}$, where $t_i = \max\{r_i, s_i\}$ for $i = 1, \dots , k$.
\end{proposition}
\begin{proof}
We prove by induction on $k$. If ${\mathcal{P}}_{M}(N) = {\pr p}^{r}$ and ${\mathcal{P}}_{M'}(N') = {\pr p}^{s}$ for some prime ideal $\pr p$. Without loss of generality, we assume $r \geq s$. Then we have RPE filtrations
$$N \overset{\pr p} \subset N_1 \overset{\pr p} \subset N_2 \subset \cdots \subset N_{s} \overset{\pr p} \subset N_{s +1} \subset \cdots \subset N_{r} = M$$
$$N' \overset{\pr p} \subset N_1' \overset{\pr p} \subset N_2' \subset \cdots \subset N_{s}'  = M'.$$
Then by Lemma \ref{directsumlemma},
\begin{multline*}
N \oplus N' \overset{\pr p} \subset N_1 \oplus N_1' \overset{\pr p} \subset N_2 \oplus N_2' \overset{\pr p}\subset \cdots \overset{\pr p}\subset N_{s}\oplus N_{s}' = N_{s}\oplus M' \overset{\pr p} \subset N_{s +1}\oplus M' \\ \overset{\pr p}\subset \cdots \overset{\pr p}\subset N_{r}\oplus M' = M\oplus M'
\end{multline*}
is an RPE filtration. Therefore, ${\mathcal{P}}_{M \oplus M'}(N \oplus N') = {\pr p}^{r}$. Hence, the result is true for $k=1$.

Now let $k > 1$. By reordering ${{\pr p}_{1}}, \dots , {{\pr p}_{k}}$ we assume that ${\pr p}_i \not \subset {\pr p}_j$ whenever $i<j$. Then ${\pr p}_1$ is maximal in $\assr{M/N} \cup \assr{M'/N'}$. By Remark \ref{reordering}, there exists RPE filtrations
$$N \overset{{\pr p}_1} \subset N_1 \overset{{\pr p}_1} \subset N_2 \subset \cdots \overset{{\pr p}_1} \subset N_{r_1} \overset{{\pr p}_2} \subset K_1 \subset \cdots \subset M$$
$$N' \overset{{\pr p}_1} \subset N_1' \overset{{\pr p}_1} \subset N_2' \subset \cdots \overset{{\pr p}_1} \subset N_{s_1}' \overset{{\pr p}_2} \subset K_1' \subset \cdots \subset M'.$$
So we have ${\mathcal{P}}_{M}(N_{r_1}) = {{\pr p}_2}^{r_2} \cdots {{\pr p}_k}^{r_k}$ and ${\mathcal{P}}_{M'}(N_{s_1}') = {{\mathfrak{p}}_2}^{s_2} \cdots {{\mathfrak{p}}_{k}}^{s_k}$. By induction assumption, ${\mathcal{P}}_{M \oplus M'}(N_{r_1} \oplus N_{s_1}') = {{\pr p}_2}^{t_2} \cdots {{\pr p}_k}^{t_k}$, where $t_i = \max\{r_i, s_i\}$ for $i = 2, \dots , k$. By the previous paragraph, we have ${\mathcal{P}}_{N_{r_1} \oplus N_{s_1}'}(N \oplus N') = {{\pr p}_1}^{t_1}$, where $t_1 = \max\{r_1, s_1\}$. Then by Lemma \ref{directsum_corollary}, $N_{r_1} \oplus N_{s_1}'$ lies in an RPE filtration of $M \oplus M'$ over $N \oplus N'$. So by Lemma \ref{reg_div_rpe}, we get ${\mathcal{P}}_{M \oplus M'}(N \oplus N') = {\mathcal{P}}_{M \oplus M'}(N_{r_1} \oplus N_{s_1}')$ $ {\mathcal{P}}_{N_{r_1} \oplus N_{s_1}'}(N \oplus N') = {{\mathfrak{p}}_1}^{t_1} \cdots {{\mathfrak{p}}_{k}}^{t_k}$.
\end{proof}
Using induction, we have the following theorem.
\begin{theorem}\label{directsumthm}
Let $N_1, \dots , N_n$ be proper submodules of $R$-modules $M_1, \dots , M_n$ respectively. If ${\mathcal{P}}_{M_i}(N_i) = {{\pr p}_{1}}^{r_{i_1}} \cdots {{\pr p}_{k}}^{r_{i_k}} $ for $i = 1, \dots, n$ and $r_{i_j} \geq 0$, then ${\mathcal{P}}_{\oplus_{i=1}^n{M}_i}( \bigoplus\limits_{i=1}^n{N}_i ) = {{\pr p}_{1}}^{s_1} \cdots {{\pr p}_{k}}^{s_k}$, where $s_j = \max\{r_{1_j}, \dots , r_{n_j}\}$ for $j = 1, \dots , k$.
\end{theorem}

Now we get the following sufficient condition for the existence of an $R$-module $M$ and a submodule $N$ of $M$ with ${\mathcal{P}}_M(N) = \qqq n$.

\begin{corollary}
Let ${\pr p}_1, \dots, {\pr p}_n$ be distinct prime ideals in $R$ and $r_1, \dots , r_n$ be positive integers. If ${{\pr p}_i}^{r_i} \neq {{\pr p}_i}^{r_i - 1}$ and $\assr{R/{{\pr p}_i}^{r_i}} = \{{{\pr p}_i}\}$ for every $1 \leq i \leq n$, then ${\mathcal{P}}_{R^n}( {{\pr p}_1}^{r_1} \oplus \cdots \oplus {{\pr p}_n}^{r_n} ) = \qqq n$.
\end{corollary}
\begin{proof}
By Corollary \ref{cor_power_iff}, ${\mathcal{P}}_R({{\pr p}_i}^{r_i}) = {{\pr p}_i}^{r_i}$ for every $1 \leq i \leq n$. Then by Theorem \ref{directsumthm}, ${\mathcal{P}}_{R^n}( {{\pr p}_1}^{r_1} \oplus \cdots \oplus {{\pr p}_n}^{r_n} ) = \qqq n$.
\end{proof}

The next theorem gives a necessary and sufficient condition for the existence of an $R$-module $M$ with a submodule $N$ such that ${\mathcal{P}}_M(N) = \qqq n$.

\begin{theorem}\label{directsumiff}
Let ${\pr p}_1, \dots, {\pr p}_n$ be distinct prime ideals in $R$ and $r_1, \dots , r_n$ be positive integers. There exists an $R$-module $M$ and a submodule $N$ of $M$ with ${\mathcal{P}}_{M}(N)=\qqq n$ if and only if there exist $R$-modules $M_i$ and submodules $N_i$ of $M_i$ such that ${\mathcal{P}}_{M_i}(N_i) = {{\pr p}_i}^{r_i}$, $1 \leq i \leq n$.
\end{theorem}
\begin{proof}
If there exist $R$-modules $M_i$ and submodules $N_i$ of $M_i$ such that ${\mathcal{P}}_{M_i}(N_i) = {{\pr p}_i}^{r_i}$, $1 \leq i \leq n$, then by Theorem \ref{directsumthm}, for $N={\oplus_{i=1}^n{N}_i}$ and $M={\oplus_{i=1}^n{M}_i}$ we have ${\mathcal{P}}_{M}(N)=\qqq n$.

Conversely if there exists an $R$-module $M$ and a submodule $N$ of $M$ with ${\mathcal{P}}_M(N) = \qqq n$, then taking $M_i = L^{(i)}_{r_{i}}$ and $N_i = L^{(i-1)}_{r_{i-1}}$ in the RPE filtration 
\begin{multline*}
N= L^{(0)} \overset{{\pr p}_{1}}\subset L^{(1)}_1 \overset{{\pr p}_{1}}\subset L^{(1)}_2 \overset{{\pr p}_{1}}\subset \cdots \overset{{\pr p}_{1}}\subset L^{(1)}_{r_1} \overset{{\pr p}_{2}}\subset L^{(2)}_1 \subset \cdots \\ \overset{{\pr p}_{n-1}}\subset L^{(n-1)}_{r_{n-1}} \overset{{\pr p}_{n}}\subset L^{(n)}_1 \overset{{\pr p}_{n}}\subset L^{(n)}_2 \subset \cdots \overset{{\pr p}_{n}}\subset L^{(n)}_{r_{n}} = M,
\end{multline*}
we get ${\mathcal{P}}_{M_i}(N_i) = {{\pr p}_i}^{r_i}$, $1 \leq i \leq n$.
\end{proof}


%
\vskip 0.4 true cm

\bibliographystyle{amsplain}

\bigskip
\bigskip

\noindent {\footnotesize {\bf  K. R. Thulasi }\; \\
Department of Mathematics, Pondicherry University, Pondicherry, India.\\
 {\tt thulasi.3008@gmail.com}

\noindent {\footnotesize {\bf  T. Duraivel }\; \\
Department of Mathematics, Pondicherry University, Pondicherry, India.\\
 {\tt tduraivel@gmail.com}
 
\noindent {\footnotesize {\bf  S. Mangayarcarassy }\; \\
Department of Mathematics, Puducherry Technological University, Pondicherry, India.\\
 {\tt dmangay@pec.edu}

\end{document}